\documentclass[12pt,a4paper,reqno]{amsart}
\usepackage{amsmath}
\usepackage{amsthm}
\usepackage{amsfonts}
\usepackage{amssymb}
\usepackage{color}
\usepackage{mathrsfs}
\usepackage{graphicx}

\numberwithin{equation}{section}

\theoremstyle{plain}
\newtheorem{thm}{Theorem}

\theoremstyle{plain}
\newtheorem{lem}{Lemma}

\theoremstyle{plain}
\newtheorem*{theorema}{Theorem~A}

\theoremstyle{remark}
\newtheorem*{remark}{Remark}

\theoremstyle{definition}
\newtheorem*{casea}{Case~A}

\theoremstyle{definition}
\newtheorem*{caseb}{Case~B}

\theoremstyle{definition}
\newtheorem*{case1}{Case~1}

\theoremstyle{definition}
\newtheorem*{case2}{Case~2}

\theoremstyle{definition}

\theoremstyle{definition}

\theoremstyle{definition}

\def\bfiti{\boldsymbol{i}}

\def\eps{\varepsilon}

\def\Rr{\mathbb{R}}

\def\III{\mathcal{I}}

\def\PPP{\mathcal{P}}

\def\XXX{\mathcal{X}}
\def\YYY{\mathcal{Y}}

\def\frakI{\mathfrak{I}}

\def\LLLL{\mathscr{L}}

\renewcommand{\le}{\leqslant}
\renewcommand{\ge}{\geqslant}

\title[Superdensity and super-micro-uniformity]
{Superdensity and super-micro-uniformity\\
in non-integrable flat systems}

\author[Beck]{J. Beck}

\address{Department of Mathematics, Hill Center for the Mathematical Sciences,
Rutgers University, Piscataway NJ 08854, USA}

\email{jbeck@math.rutgers.edu}

\author[Chen]{W.W.L. Chen}

\address{School of Mathematical and Physical Sciences, Faculty of Science and Engineering,
Macquarie University, Sydney NSW 2109, Australia}

\email{william.chen@mq.edu.au}



\begin{document}

\begin{abstract}
We show that on any non-integrable finite polysquare translation surface, superdensity,
an optimal form of time-quantitative density, leads to an optimal form of time-quantitative uniformity
that we call super-micro-uniformity.
\end{abstract}

\maketitle

\thispagestyle{empty}

%%%%%%%%%%
%
% SECTION 1
%
%%%%%%%%%%

\section{Introduction}\label{sec1}

Consider a half-infinite geodesic on a finite polysquare translation surface.
It is trivial that uniformity always implies density, and that the converse is false.
However, while density does not in general imply uniformity, we demonstrate here an interesting case
when some form of time-quantitative density implies some form of time-quantitative uniformity.

The purpose of the present paper is to show how superdensity, an optimal form of time-quantitative density,
implies an optimal form of time-quantitative uniformity that we call \textit{super-micro-uniformity}.
Here \textit{super} refers to optimality and \textit{micro} refers to microscopic scale.

To illustrate the latter, consider the irrational rotation sequence
\begin{equation}\label{eq1.1}
\{q\alpha\},
\quad
q=1,2,3,\ldots,
\end{equation}
of fractional parts of $q\alpha$ in the interval $[0,1)$, where $\alpha$ is irrational.
Let $I\subset[0,1)$ be an arbitrary subinterval of length~$1/2n$,
and consider the first $n$ elements of the sequence \eqref{eq1.1}.
Then the \textit{expected number} of elements of this $n$-element set in $I$ is clearly~$1/2$,
corresponding to $n$ times the length of~$I$.
On the other hand, the \textit{visiting number} $V_n(I)$ of~$I$, the actual number of elements in $I$
coming from this $n$-element set is clearly an integer, and so must differ from the expected number by at least~$1/2$.
We refer to this as the \textit{trivial error}.
Indeed, we have same phenomenon if the length of $I$ is~$C/n$, where $2C$ is an odd integer.
Here the error is at least~$1/2$, and the expected number $C$ is in the constant range.

Given the first $n$ elements of the infinite sequence \eqref{eq1.1},
intervals of length $C/n$ represent test sets in the microscopic scale.
Here $C>0$ is a fixed constant, and $n$ may tend to infinity.
The trivial error argument above implies that in the microscopic scale of~$C/n$,
we cannot expect \textit{perfect local uniformity} in the sense that the ratio of the error term
and the expected number tends to zero as $C$ is fixed and $n$ tends to infinity.
To put it slightly differently, to have perfect local uniformity, it is necessary to have $C=C(n)\to\infty$ as $n\to\infty$.

It turns out that this necessary condition is sufficient to establish perfect local uniformity if $\alpha$ is badly approximable.
This perfect local uniformity is what we call super-micro-uniformity.
It has the intuitive meaning that the orbit exhibits uniformity already in the shortest possible subintervals.
We have the following result on super-micro-uniformity of the irrational rotation sequence generated by a badly approximable~$\alpha$.

\begin{theorema}
Let $\alpha$ be a badly approximable real number.
For any  subinterval $I\subset[0,1)$, let
\textcolor{white}{xxxxxxxxxxxxxxxxxxxxxxxxxxxxxx}%%%%%%%%%%
\begin{displaymath}
V_n(I)=\vert\{q=1,\ldots,n:\{q\alpha\}\in I\}\vert
\end{displaymath}
denote the visiting number of $I$ with respect to the first $n$ terms of the sequence \eqref{eq1.1}.
Then for every sufficiently large integer $n$ and
every real number $\eps>0$, there exists a finite threshold $C_\eps=C_\eps(\alpha)$
satisfying $1<C_\eps<n$
such that for any  subinterval $I$ with length $\vert I\vert\ge C_\eps/n$, the inequality
\begin{equation}\label{eq1.2}
\vert V_n(I)-n\vert I\vert\vert<\eps n\vert I\vert
\end{equation}
holds.
\end{theorema}

The proof of this result is a fairly routine exercise using continued fractions, so Theorem~A is very possibly folklore.
However, as we shall establish a more general result, we briefly outline the ideas here.

First of all, we recall that the convergents $p_k/q_k$ of $\alpha$ give excellent rational approximation, in the sense that
\begin{displaymath}
\left\vert\alpha-\frac{p_k}{q_k}\right\vert<\frac{1}{q_kq_{k+1}},
\end{displaymath}
so that
\textcolor{white}{xxxxxxxxxxxxxxxxxxxxxxxxxxxxxx}%%%%%%%%%%
\begin{displaymath}
\left\vert q\alpha-\frac{qp_k}{q_k}\right\vert<\frac{q}{q_kq_{k+1}}\le\frac{1}{q_{k+1}},
\quad
q=1,\ldots,q_k.
\end{displaymath}
Hence any segment of $q_k$ consecutive terms of the sequence \eqref{eq1.1}
is very uniformly distributed in the interval $[0,1)$.

To take advantage of this, it makes sense to look at the Ostrowski decomposition of integers,
using the denominators of the convergents.
For every integer~$N$, we can write
\textcolor{white}{xxxxxxxxxxxxxxxxxxxxxxxxxxxxxx}%%%%%%%%%%
\begin{displaymath}
N=\sum_{i=0}^mb_kq_k,
\end{displaymath}
where $m$ is the unique integer satisfying $q_m\le N<q_{m+1}$, and the digits $b_0,b_1,\ldots,b_m$ satisfy
\textcolor{white}{xxxxxxxxxxxxxxxxxxxxxxxxxxxxxx}%%%%%%%%%%
\begin{displaymath}
\begin{array}{ll}
b_0\in\{0,1,\ldots,a_1-1\},
\vspace{2pt}
\\
b_k\in\{0,1,\ldots,a_{k+1}\},
&
k=1,\ldots,m,
\vspace{2pt}
\\
b_{k-1}=0\mbox{ if }b_k=a_{k+1},
&
k=1,\ldots,m.
\end{array}
\end{displaymath}
where $a_1,\ldots,a_{m+1}$ are continued fraction digits of~$\alpha$.

Theorem~A follows on combining these two ideas in a suitable way.

From the discrete super-micro-uniformity given by \eqref{eq1.2}, it is easy to deduce that every half-infinite geodesic,
\textit{i.e.} torus line, of badly approximable slope $\alpha$ is super-micro-uniform in the unit torus $[0,1)^2$.

Note that geodesics on the unit torus $[0,1)^2$ is the simplest integrable system.
If we consider geodesic flow on an arbitrary finite polysquare translation surface, then it is typically non-integrable.

\begin{thm}\label{thm1}
Let $\PPP$ be a polysquare translation surface with $b$ atomic squares, and let $\alpha$ be a badly approximable real number.
Let $L_\alpha(t)$, $t\ge0$, be a half-infinite geodesic with slope~$\alpha$, equipped with the usual arc-length parametrization.
For any positive integer~$n$, let $\XXX_n$ denote the set of the first $n$ intersection points of $L_\alpha(t)$, $t\ge0$,
with the vertical edges of~$\PPP$, and for any  subinterval $I$ of any vertical edge of~$\PPP$, let
\begin{displaymath}
V_n(I)=\vert I\cap\XXX_n\vert
\end{displaymath}
denote the visiting number of $I$ with respect to~$\XXX_n$.
Then for every sufficiently large integer $n$ and every real number $\eps>0$,
there exists a finite threshold $C_\eps=C_\eps(\PPP;\alpha)$
satisfying $1<C_\eps<n$
such that for any  subinterval $I$ of any vertical edge of
$\PPP$ with length $\vert I\vert\ge C_\eps/n$, the inequality
\begin{displaymath}
\left\vert V_n(I)-\frac{n\vert I\vert}{b}\right\vert<\eps\frac{n\vert I\vert}{b}
\end{displaymath}
holds.
In other words, we have super-micro-uniformity.
\end{thm}

The remainder of the paper is devoted to proving this result.

Needless to say, super-micro-uniformity implies traditional Weyl type uniformity with respect to all Jordan measurable test sets.

We require a superdensity result in our earlier papers \cite{BC1,BC2}.
Let $\PPP$ be a polysquare translation surface with $b$ atomic squares, and let $\alpha$ be a badly approximable real number.
Then there exists a finite superdensity threshold $c_0=c_0(\PPP;\alpha)$ such that for every integer $m\ge1$,
any geodesic segment of slope $\alpha$ and length $c_0m$ gets $(1/m)$-close to every point of~$\PPP$.

%%%%%%%%%%
%
% SECTION 2
%
%%%%%%%%%%

\section{Iteration process: step 0}\label{sec2}

Let $C$ be a constant satisfying $1<C<n$.
Let $\III_n(\PPP;C)$ denote the collection of all  subintervals $I$ of any vertical edge of $\PPP$ with length $\vert I\vert=C/n$,
and let $I_0,I_1\in\III_n(\PPP;C)$ be  subintervals satisfying
\begin{displaymath}
V_n(I_0)=\min_{I\in\III_n(\PPP;C)}\vert I\cap\XXX_n\vert
\quad\mbox{and}\quad
V_n(I_1)=\max_{I\in\III_n(\PPP;C)}\vert I\cap\XXX_n\vert,
\end{displaymath}
so that $I_0$ and $I_1$ have respectively the smallest and largest visiting numbers with respect to $\XXX_n$
among all the  subintervals $I$ under consideration.
It is clear that
\begin{equation}\label{eq2.1}
\vert I_0\cap\XXX_n\vert\le\frac{C}{b}\le\vert I_1\cap\XXX_n\vert.
\end{equation}
Let the real number $\eps$ satisfy $0<\eps<1/2$.
We have two cases:

\begin{casea}
We have
\textcolor{white}{xxxxxxxxxxxxxxxxxxxxxxxxxxxxxx}%%%%%%%%%%
\begin{equation}\label{eq2.2}
\frac{\vert I_0\cap\XXX_n\vert}{\vert I_1\cap\XXX_n\vert}\ge1-\eps.
\end{equation}
\end{casea}

\begin{caseb}
We have
\textcolor{white}{xxxxxxxxxxxxxxxxxxxxxxxxxxxxxx}%%%%%%%%%%
\begin{equation}\label{eq2.3}
\frac{\vert I_0\cap\XXX_n\vert}{\vert I_1\cap\XXX_n\vert}<1-\eps.
\end{equation}
\end{caseb}

We shall postpone the analysis of Case~A to Section~\ref{sec6}.

To complete the proof of Theorem~\ref{thm1}, we shall show that Case~B, where \eqref{eq2.3} holds, is not possible.
Indeed, we shall show that \eqref{eq2.3} leads to a contradiction.
The proof is rather long, and involves a complicated iteration process, with two possibilities at each step.
We shall derive the necessary contradiction by showing that at some stage of this process, neither possibility is valid.

We need the following number theoretic technical result.

\begin{lem}\label{lem1}
Suppose that $q_k$ is the denominator of a convergent of~$\alpha$,
and that $I$ is an interval of real numbers with length $\vert I\vert\le1/2q_k$.
Then at most one of the translated intervals
\textcolor{white}{xxxxxxxxxxxxxxxxxxxxxxxxxxxxxx}%%%%%%%%%%
\begin{equation}\label{eq2.4}
I+q\alpha,
\quad
q=1,\ldots,q_k,
\end{equation}
contains an integer.
\end{lem}

\begin{proof}
Consider the $q_k$ points $\{q\alpha\}$, $q=1,\ldots,q_k$.
It follows from a special case of the famous $3$-distance theorem \cite{sos57,sos58} that
the distance between two neighbouring points of this finite collection of numbers is at least
\begin{displaymath}
\Vert q_{k-1}\alpha\Vert\ge\frac{1}{q_k+q_{k-1}}>\frac{1}{2q_k}.
\end{displaymath}
This implies that if $\vert I\vert\le1/2q_k$,
then the $q_k$ translated intervals \eqref{eq2.4} are pairwise disjoint modulo~$1$,
so that at most one contains an integer.
\end{proof}

We also need the following counting result.

\begin{lem}\label{lem2}
Let $\alpha$ be a badly approximable number, and let $A$ be an upper bound on the continued fraction digits of~$\alpha$.
Consider a set
\begin{displaymath}
\YYY_m=\{\{\beta+q\alpha\}:q=1,\ldots,m\}\subset[0,1),
\end{displaymath}
where $m$ is a positive integer, $\beta\in\Rr$ is arbitrary and the interval $I^\star\subset[0,1)$.
Then
\begin{equation}\label{eq2.5}
\vert I^\star\cap\YYY_m\vert\le2(A+1)m\vert I^\star\vert+2.
\end{equation}
\end{lem}

\begin{proof}
Suppose that
\textcolor{white}{xxxxxxxxxxxxxxxxxxxxxxxxxxxxxx}%%%%%%%%%%
\begin{equation}\label{eq2.6}
q_{k-1}<m\le q_k,
\end{equation}
where $q_{k-1}$ and $q_k$ are the denominators of successive convergents of~$\alpha$.
We expand the set $\YYY_m$ to the set
\begin{displaymath}
\YYY_{q_k}=\{\{\beta+q\alpha\}:q=1,\ldots,q_k\}\subset[0,1),
\end{displaymath}
which has good distribution properties in $[0,1)$.
Clearly
\begin{equation}\label{eq2.7}
\vert I^\star\cap\YYY_m\vert\le\vert I^\star\cap\YYY_{q_k}\vert,
\end{equation}
so we need to find an upper bound for the right hand side.
Using a special case of the $3$-distance theorem,
we know that the distance between neighbouring points of the set $\YYY_{q_k}$ is equal to
\begin{equation}\label{eq2.8}
\Vert q_{k-1}\alpha\Vert
\quad\mbox{or}\quad
\Vert q_{k-1}\alpha\Vert+\Vert q_k\alpha\Vert<2\Vert q_{k-1}\alpha\Vert.
\end{equation}
Thus a generous upper bound is given by
\begin{equation}\label{eq2.9}
\vert I^\star\cap\YYY_{q_k}\vert\le2q_k\vert I^\star\vert+2,
\end{equation}
where the first factor $2$ covers for the different lengths \eqref{eq2.8} of the gaps between neighbouring points of~$\YYY_{q_k}$,
and the second factor $2$ covers for any error arising from the two endpoints of the interval~$I^\star$.
The estimate \eqref{eq2.5} now follows on combining \eqref{eq2.6}, \eqref{eq2.7}, \eqref{eq2.9} and
the trivial estimate $q_k<(A+1)q_{k-1}$.
\end{proof}

Let $q_k$ be the smallest convergent denominator such that
\begin{equation}\label{eq2.10}
q_k(1+\alpha^2)^{1/2}>\frac{6c_0n}{C}.
\end{equation}
Then
\textcolor{white}{xxxxxxxxxxxxxxxxxxxxxxxxxxxxxx}%%%%%%%%%%
\begin{displaymath}
q_{k-1}\le\frac{6c_0n}{C(1+\alpha^2)^{1/2}},
\end{displaymath}
and so
\textcolor{white}{xxxxxxxxxxxxxxxxxxxxxxxxxxxxxx}%%%%%%%%%%
\begin{equation}\label{eq2.11}
q_k\le(A+1)q_{k-1}\le\frac{6(A+1)c_0n}{C(1+\alpha^2)^{1/2}},
\end{equation}
where $A$ is an upper bound on the continued fraction digits of~$\alpha$.
We divide the interval $I_0$ into subintervals $\frakI$ of common length
\begin{equation}\label{eq2.12}
\vert\frakI\vert=\frac{1}{2q_k}<\frac{C(1+\alpha^2)^{1/2}}{12c_0n}\le\frac{C}{3n},
\end{equation}
provided that
\textcolor{white}{xxxxxxxxxxxxxxxxxxxxxxxxxxxxxx}%%%%%%%%%%
\begin{equation}\label{eq2.13}
c_0\ge\frac{(1+\alpha^2)^{1/2}}{4},
\end{equation}
and ignore the short remainder.

There is no problem with satisfying the requirement \eqref{eq2.13}, as we simply increase the superdensity threshold constant $c_0$ if necessary.
The inequalities in \eqref{eq2.12} and \eqref{eq2.13} are vital, since otherwise the intervals $\frakI$ would be too long to be subintervals of~$I_0$.

Superdensity implies that a geodesic segment with slope $\alpha$ and length $6c_0n/C$ visits the middle third of~$I_1$,
and ensures also that a geodesic flow with slope $\alpha$ and length $6c_0n/C$ sweeps any subinterval~$\frakI$, in view of \eqref{eq2.12},
to a union of subintervals in $I_1$ but not necessarily in the middle third of~$I_1$.
Combining this with Lemma~\ref{lem1} and \eqref{eq2.10},
we see that a geodesic flow with slope $\alpha$ and length $6c_0n/C$ sweeps each $\frakI$
with at most one splitting to a union of at most two subintervals in~$I_1$.
Denote by $I_1(0)$ the longest  subinterval in $I_1$ arising as part of an image of the geodesic flow in this process,
and let $I_0(0)$ denote the pre-image of $I_1(0)$ in~$I_0$.
Then
\begin{equation}\label{eq2.14}
\frac{C}{3n}
=\frac{\vert I_0\vert}{3}
\ge\vert I_0(0)\vert
=\vert I_1(0)\vert
\ge\frac{1}{4q_k}\ge\frac{C(1+\alpha^2)^{1/2}}{24(A+1)c_0n}
=c_1\vert I_0\vert,
\end{equation}
where
\textcolor{white}{xxxxxxxxxxxxxxxxxxxxxxxxxxxxxx}%%%%%%%%%%
\begin{displaymath}
c_1=c_1(\alpha)=\frac{(1+\alpha^2)^{1/2}}{24(A+1)c_0}.
\end{displaymath}
Note that the number of vertical edges hit by the geodesic flow with slope $\alpha$ from $I_0(0)$ to $I_1(0)$ is bounded above by \eqref{eq2.11},
and the length of the flow is bounded by
\begin{equation}\label{eq2.15}
q_k(1+\alpha^2)^{1/2}\le\frac{6(A+1)c_0n}{C}.
\end{equation}

We have two cases:

\begin{case1}
We have
\textcolor{white}{xxxxxxxxxxxxxxxxxxxxxxxxxxxxxx}%%%%%%%%%%
\begin{equation}\label{eq2.16}
\frac{\vert I_1(0)\cap\XXX_n\vert}{\vert I_1(0)\vert}\ge\left(1-\frac{\eps}{2}\right)\frac{\vert I_1\cap\XXX_n\vert}{\vert I_1\vert}.
\end{equation}
\end{case1}

\begin{case2}
We have
\textcolor{white}{xxxxxxxxxxxxxxxxxxxxxxxxxxxxxx}%%%%%%%%%%
\begin{equation}\label{eq2.17}
\frac{\vert I_1(0)\cap\XXX_n\vert}{\vert I_1(0)\vert}<\left(1-\frac{\eps}{2}\right)\frac{\vert I_1\cap\XXX_n\vert}{\vert I_1\vert}.
\end{equation}
\end{case2}

Before we study these cases in detail, we first give some heuristics to explain the underlying ideas.
If Case~1 holds, then we show that the subinterval $I_0(0)\subset I_0$ exhibits a surplus density of points of $\XXX_n$ compared to~$I_0$.
Removing this subinterval, the remaining part of $I_0$ then exhibits a deficit density of points of $\XXX_n$ compared to~$I_0$.
If Case~2 holds, then the subinterval $I_1(0)\subset I_1$ exhibits a deficit density of points of $\XXX_n$ compared to~$I_1$.
Removing this subinterval, the remaining part of $I_1$ then exhibits a surplus density of points of $\XXX_n$ compared to~$I_1$.
Thus we either obtain a subinterval of $I_0$ that exhibits deficit density of points of $\XXX_n$ compared to~$I_0$,
or obtain a subinterval of $I_1$ that exhibits surplus density of points of $\XXX_n$ compared to~$I_1$.
We then repeat the analysis on intervals of length equal to the length of this subinterval, and this sets up an iteration process.
We then show that if Case~B holds, then this iteration has to terminate after a finite number of steps,
and this gives the necessary contradiction.

%%%%%%%%%%
%
% SECTION 2.1
%
%%%%%%%%%%

\subsection{Case 1: density decrease}\label{sec21}

Suppose that the inequality \eqref{eq2.16} holds.
To find a lower bound to $\vert I_0(0)\cap\XXX_n\vert$, we consider the transportation process
as the geodesic flow with slope $\alpha$ moves the interval $I_0(0)$ to the interval $I_1(0)$.
Let
\begin{displaymath}
\XXX_n=\{x_1,\ldots,x_n\},
\end{displaymath}
and let $n^*$ denote the number of times that this finite transportation process from $I_0(0)$ to $I_1(0)$
intersects a vertical edge of~$\PPP$,
so that $I_1$ is the $n^*$-th vertical edge in this process.
Suppose that $j=1,\ldots,n$ and the point $x_j\in I_1(0)$, contributing a count of $1$ to $\vert I_1(0)\cap\XXX_n\vert$.
Then provided that $j-n^*>0$, the point $x_{j-n^*}\in I_0(0)$, contributing a count of $1$ to $\vert I_0(0)\cap\XXX_n\vert$.
On the other hand, if $j\le n^*$, then while the point $x_j\in I_1(0)$ contributes a count of $1$ to $\vert I_1(0)\cap\XXX_n\vert$,
there is no corresponding contribution to $\vert I_0(0)\cap\XXX_n\vert$.
In other words,
\begin{equation}\label{eq2.18}
\vert I_1(0)\cap\XXX_n\vert-\vert I_0(0)\cap\XXX_n\vert\le\vert I_1(0)\cap\XXX_{n^*}\vert,
\end{equation}
where $\XXX_{n^*}$ is the collection of the first $n^*$ intersection points in~$\XXX_n$.
Since geodesic flow on $\PPP$ modulo~$1$ is geodesic flow on the unit torus,
and the slope $\alpha$ is badly approximable with continued fraction digit upper bound~$A$, 
it follows from Lemma~\ref{lem2} that
\textcolor{white}{xxxxxxxxxxxxxxxxxxxxxxxxxxxxxx}%%%%%%%%%%
\begin{equation}\label{eq2.19}
\vert I_1(0)\cap\XXX_{n^*}\vert\le2(A+1)n^*\vert I_1(0)\vert+2.
\end{equation}
On the other hand, we have
\textcolor{white}{xxxxxxxxxxxxxxxxxxxxxxxxxxxxxx}%%%%%%%%%%
\begin{equation}\label{eq2.20}
n^*\le q_k\le\frac{6(A+1)c_0n}{C(1+\alpha^2)^{1/2}}.
\end{equation}
Combining \eqref{eq2.18}--\eqref{eq2.20}, we deduce that
\begin{equation}\label{eq2.21}
\vert I_0(0)\cap\XXX_n\vert
\ge\vert I_1(0)\cap\XXX_n\vert-\frac{12(A+1)^2c_0n}{C(1+\alpha^2)^{1/2}}\vert I_1(0)\vert-2.
\end{equation}
Note that \eqref{eq2.16} gives
\begin{equation}\label{eq2.22}
\vert I_1(0)\cap\XXX_n\vert
\ge\left(1-\frac{\eps}{2}\right)\vert I_1(0)\vert\frac{\vert I_1\cap\XXX_n\vert}{\vert I_1\vert}.
\end{equation}
Combining \eqref{eq2.21} and \eqref{eq2.22} and noting that $\vert I_1\cap\XXX_n\vert\ge C/b$, $\vert I_1\vert=C/n$
and $\vert I_0(0)\vert=\vert I_1(0)\vert$, we obtain the inequality
\begin{equation}\label{eq2.23}
\vert I_0(0)\cap\XXX_n\vert
\ge\left(1-\frac{3\eps}{4}\right)\vert I_0(0)\vert\frac{\vert I_1\cap\XXX_n\vert}{\vert I_1\vert},
\end{equation}
provided that
\begin{displaymath}
\frac{12(A+1)^2c_0n}{C(1+\alpha^2)^{1/2}}\le\frac{\eps}{8}\frac{\vert I_1\cap\XXX_n\vert}{\vert I_1\vert}
\quad\mbox{and}\quad
2\le\frac{\eps}{8}\vert I_0(0)\vert\frac{\vert I_1\cap\XXX_n\vert}{\vert I_1\vert},
\end{displaymath}
and these are guaranteed if we ensure that
\begin{equation}\label{eq2.24}
C\ge\frac{96(A+1)^2c_0b}{\eps(1+\alpha^2)^{1/2}}
\quad\mbox{and}\quad
C\ge\frac{16b}{c_1\eps},
\end{equation}
the latter inequality in view of \eqref{eq2.14}.
Combining \eqref{eq2.3} and \eqref{eq2.23} now leads to the inequality
\textcolor{white}{xxxxxxxxxxxxxxxxxxxxxxxxxxxxxx}%%%%%%%%%%
\begin{align}\label{eq2.25}
\vert I_0(0)\cap\XXX_n\vert
&
\ge\left(1-\frac{3\eps}{4}\right)(1-\eps)^{-1}\vert I_0(0)\vert\frac{\vert I_0\cap\XXX_n\vert}{\vert I_0\vert}
\nonumber
\\
&
>\left(1+\frac{\eps}{4}\right)\vert I_0(0)\vert\frac{\vert I_0\cap\XXX_n\vert}{\vert I_0\vert},
\end{align}
provided that \eqref{eq2.24} holds.

There are at most two subintervals $I_{0,1},I_{0,2}\subset I_0$ such that
\begin{equation}\label{eq2.26}
I_0=I_0(0)\cup I_{0,1}\cup I_{0,2}.
\end{equation}
Removing the interval $I_0(0)$ and combining \eqref{eq2.25} and \eqref{eq2.26}, we obtain
\begin{align}\label{eq2.27}
&
\vert I_{0,1}\cap\XXX_n\vert+\vert I_{0,2}\cap\XXX_n\vert
=\vert I_0\cap\XXX_n\vert-\vert I_0(0)\cap\XXX_n\vert
\nonumber
\\
&\quad
<\vert I_0\cap\XXX_n\vert\left(1-\left(1+\frac{\eps}{4}\right)\frac{\vert I_0(0)\vert}{\vert I_0\vert}\right)
\le\vert I_0\cap\XXX_n\vert\left(\frac{\vert I_{0,1}\vert+\vert I_{0,2}\vert}{\vert I_0\vert}-\frac{c_1\eps}{4}\right),
\end{align}
noting that \eqref{eq2.14} implies $\vert I_0(0)\vert/\vert I_0\vert\ge c_1$.
There are two possibilities, either
\begin{equation}\label{eq2.28}
\frac{\min\{\vert I_{0,1}\vert,\vert I_{0,2}\vert\}}{\vert I_0\vert}<\frac{c_1\eps}{8},
\end{equation}
or
\textcolor{white}{xxxxxxxxxxxxxxxxxxxxxxxxxxxxxx}%%%%%%%%%%
\begin{equation}\label{eq2.29}
\frac{\min\{\vert I_{0,1}\vert,\vert I_{0,2}\vert\}}{\vert I_0\vert}\ge\frac{c_1\eps}{8}.
\end{equation}
If \eqref{eq2.28} holds, then we may assume without loss of generality that
\begin{displaymath}
\vert I_{0,1}\vert\ge\vert I_{0,2}\vert,
\quad\mbox{so that}\quad
\frac{\vert I_{0,2}\vert}{\vert I_0\vert}<\frac{c_1\eps}{8},
\end{displaymath}
and so it follows from \eqref{eq2.27} that
\begin{equation}\label{eq2.30}
\vert I_{0,1}\cap\XXX_n\vert
\le\vert I_0\cap\XXX_n\vert\left(\frac{\vert I_{0,1}\vert}{\vert I_0\vert}-\frac{c_1\eps}{8}\right).
\end{equation}
On the other hand, if \eqref{eq2.29} holds, then since it follows from \eqref{eq2.27} that
\begin{displaymath}
\vert I_{0,1}\cap\XXX_n\vert+\vert I_{0,2}\cap\XXX_n\vert
\le\vert I_0\cap\XXX_n\vert\left(\frac{\vert I_{0,1}\vert}{\vert I_0\vert}-\frac{c_1\eps}{8}\right)
+\vert I_0\cap\XXX_n\vert\left(\frac{\vert I_{0,2}\vert}{\vert I_0\vert}-\frac{c_1\eps}{8}\right),
\end{displaymath}
we may assume without loss of generality that \eqref{eq2.30} holds again.
Note now that \eqref{eq2.30} leads to the inequality
\textcolor{white}{xxxxxxxxxxxxxxxxxxxxxxxxxxxxxx}%%%%%%%%%%
\begin{displaymath}
\vert I_{0,1}\vert\ge\frac{c_1\eps}{8}\vert I_0\vert,
\end{displaymath}
as well as the inequality
\begin{equation}\label{eq2.31}
\frac{\vert I_{0,1}\cap\XXX_n\vert}{\vert I_{0,1}\vert}
\le\frac{\vert I_0\cap\XXX_n\vert}{\vert I_0\vert}\left(1-\frac{c_1\eps}{8}\frac{\vert I_0\vert}{\vert I_{0,1}\vert}\right)
\le\left(1-\frac{c_1\eps}{8}\right)\frac{\vert I_0\cap\XXX_n\vert}{\vert I_0\vert}.
\end{equation}
Thus the switch from $I_0$ to $I_{0,1}$ leads to density decrease by a factor of $1-c_1\eps/8$.

The ratio $\vert I_0\vert/\vert I_{0,1}\vert$ is not necessarily an integer.
To overcome this issue, we shall replace $I_{0,1}$ by a suitable subinterval at the expense of part of the density decrease.

We shall use the following almost trivial observation a number of times.

\begin{lem}\label{lem3}
Suppose that $I$ is a finite interval of real numbers, $\YYY\subset I$ is a finite subset with $m$ elements,
and $z$ is a real number satisfying $0<z\le\vert I\vert$.

\emph{(i)}
Then there exists a subinterval $I'\subset I$ of length $\vert I'\vert=z$ such that
\begin{displaymath}
\vert I'\cap\YYY\vert\le m\frac{z}{\vert I\vert-z}.
\end{displaymath}

\emph{(ii)}
Suppose further that there exists an integer $B\ge1$ such that every subinterval $I^\dagger\subset I$ satisfies
\textcolor{white}{xxxxxxxxxxxxxxxxxxxxxxxxxxxxxx}%%%%%%%%%%
\begin{equation}\label{eq2.32}
\vert I^\dagger\cap\YYY\vert\le Bm\frac{\vert I^\dagger\vert}{\vert I\vert}+2.
\end{equation}
Then there exists a subinterval $I''\subset I$ of length $\vert I''\vert=z$ such that
\begin{displaymath}
\vert I''\cap\YYY\vert\ge(m-2)\frac{z}{\vert I\vert}-Bm\left(\frac{z}{\vert I\vert}\right)^2.
\end{displaymath}
\end{lem}

\begin{proof}
Write $\vert I\vert=kz+w$, where $k\ge1$ is an integer and $0\le w<z$.
We partition the interval $I$ into a union
\begin{displaymath}
I=J_0\cup J_1\cup\ldots\cup J_k,
\quad
\vert J_0\vert=w,
\quad
\vert J_1\vert=\ldots=\vert J_k\vert=z.
\end{displaymath}

(i)
Among the intervals $J=J_1,\ldots,J_k$, let $I'$ be one for which $\vert J\cap\YYY\vert$ is minimal.
Observing the inequality
\textcolor{white}{xxxxxxxxxxxxxxxxxxxxxxxxxxxxxx}%%%%%%%%%%
\begin{displaymath}
k=\frac{\vert I\vert-w}{z}>\frac{\vert I\vert-z}{z},
\end{displaymath}
we deduce that
\textcolor{white}{xxxxxxxxxxxxxxxxxxxxxxxxxxxxxx}%%%%%%%%%%
\begin{displaymath}
\vert I'\cap\YYY\vert\le\frac{\vert\YYY\vert}{k}\le m\frac{z}{\vert I\vert-z}.
\end{displaymath}

(ii)
Among the intervals $J=J_1,\ldots,J_k$, let $I''$ be one for which $\vert J\cap\YYY\vert$ is maximal.
Applying \eqref{eq2.32} to the subinterval~$J_0$, we have
\begin{displaymath}
\vert J_0\cap\YYY\vert<\frac{Bmz}{\vert I\vert}+2.
\end{displaymath}
Observing this and the inequality
\begin{displaymath}
k=\frac{\vert I\vert-w}{z}\le\frac{\vert I\vert}{z},
\end{displaymath}
we deduce that
\begin{displaymath}
\vert I''\cap\YYY\vert
\ge\frac{\vert\YYY\vert-\vert J_0\cap\YYY\vert}{k}
\ge\frac{z}{\vert I\vert}\left(m-\frac{Bmz}{\vert I\vert}-2\right)
=(m-2)\frac{z}{\vert I\vert}-Bm\left(\frac{z}{\vert I\vert}\right)^2.
\end{displaymath}

This completes the proof.
\end{proof}

Let $h_0$ be the unique integer satisfying
\begin{equation}\label{eq2.33}
h_0-1<\frac{16\vert I_0\vert}{c_1\eps\vert I_{0,1}\vert}\le h_0,
\end{equation}
so that
\textcolor{white}{xxxxxxxxxxxxxxxxxxxxxxxxxxxxxx}%%%%%%%%%%
\begin{equation}\label{eq2.34}
\frac{\vert I_0\vert}{h_0}\le\frac{c_1\eps}{16}\vert I_{0,1}\vert\le\vert I_{0,1}\vert,
\end{equation}
provided that $\eps$ is sufficiently small.

We now apply Lemma~\ref{lem3}(i) with $I=I_{0,1}$, $\YYY=I_{0,1}\cap\XXX_n$ and $z=\vert I_0\vert/h_0$.
Then there exists a subinterval $I_0(\star)\subset I_{0,1}$ with $\vert I_0(\star)\vert=z$ such that
\begin{displaymath}
\vert I_0(\star)\cap\XXX_n\vert\le\vert I_{0,1}\cap\XXX_n\vert\frac{z}{\vert I_{0,1}\vert-z}.
\end{displaymath}
Combining this with the estimate \eqref{eq2.31}, we deduce that
\begin{equation}\label{eq2.35}
\frac{\vert I_0(\star)\cap\XXX_n\vert}{\vert I_0(\star)\vert}
\le\frac{\vert I_{0,1}\cap\XXX_n\vert}{\vert I_{0,1}\vert}\,\frac{\vert I_{0,1}\vert}{\vert I_{0,1}\vert-z}
\le\left(1-\frac{c_1\eps}{8}\right)\frac{\vert I_0\cap\XXX_n\vert}{\vert I_0\vert}\frac{\vert I_{0,1}\vert}{\vert I_{0,1}\vert-z}.
\end{equation}
Note from \eqref{eq2.34} that
\textcolor{white}{xxxxxxxxxxxxxxxxxxxxxxxxxxxxxx}%%%%%%%%%%
\begin{displaymath}
\frac{\vert I_{0,1}\vert}{\vert I_{0,1}\vert-z}
\le\left(1-\frac{c_1\eps}{16}\right)^{-1}.
\end{displaymath}
Combining this with \eqref{eq2.35}, we deduce that
\begin{equation}\label{eq2.36}
\frac{\vert I_0(\star)\cap\XXX_n\vert}{\vert I_0(\star)\vert}
\le\left(1-\frac{c_1\eps}{8}\right)\left(1-\frac{c_1\eps}{16}\right)^{-1}\frac{\vert I_0\cap\XXX_n\vert}{\vert I_0\vert}
\le\left(1-\frac{c_1\eps}{16}\right)\frac{\vert I_0\cap\XXX_n\vert}{\vert I_0\vert}.
\end{equation}
Thus the switch from $I_0$ to $I_0(\star)$ leads to density decrease by a factor of $1-c_1\eps/16$,
with the added benefit that the ratio $\vert I_0\vert/\vert I_0(\star)\vert$ is an integer~$h_0$.

To obtain a subinterval of $I_1$ of the same length as $I_0(\star)$, we next divide $I_1$ into $h_0$ equal parts,
and denote by $I_1(\star)$ one of these subintervals with the maximum intersection with the set~$\XXX_n$.
Then
\begin{equation}\label{eq2.37}
\frac{\vert I_1(\star)\cap\XXX_n\vert}{\vert I_1(\star)\vert}\ge\frac{\vert I_1\cap\XXX_n\vert}{\vert I_1\vert}.
\end{equation}
Note that
\textcolor{white}{xxxxxxxxxxxxxxxxxxxxxxxxxxxxxx}%%%%%%%%%%
\begin{equation}\label{eq2.38}
\vert I_1(\star)\vert=\vert I_0(\star)\vert=\frac{C_1}{n},
\quad
C_1=\frac{C}{h_0},
\quad
h_0<c_2,
\end{equation}
where the constant $c_2=c_2(\eps)>0$ is independent of $n$ and~$C$.

%%%%%%%%%%
%
% SECTION 2.2
%
%%%%%%%%%%

\subsection{Case 2: density increase}\label{sec22}

Suppose that the inequality \eqref{eq2.17} holds.
There are at most two subintervals $I_{1,1},I_{1,2}\subset I_1$ such that
\begin{equation}\label{eq2.39}
I_1=I_1(0)\cup I_{1,1}\cup I_{1,2}.
\end{equation}
Removing the interval $I_1(0)$ and combining \eqref{eq2.17} and \eqref{eq2.39}, we obtain
\begin{align}\label{eq2.40}
&
\vert I_{1,1}\cap\XXX_n\vert+\vert I_{1,2}\cap\XXX_n\vert
=\vert I_1\cap\XXX_n\vert-\vert I_1(0)\cap\XXX_n\vert
\nonumber
\\
&\quad
>\vert I_1\cap\XXX_n\vert\left(1-\left(1-\frac{\eps}{2}\right)\frac{\vert I_1(0)\vert}{\vert I_1\vert}\right)
\ge\vert I_1\cap\XXX_n\vert\left(\frac{\vert I_{1,1}\vert+\vert I_{1,2}\vert}{\vert I_1\vert}+\frac{c_1\eps}{2}\right),
\end{align}
noting that \eqref{eq2.14} implies $\vert I_1(0)\vert/\vert I_1\vert\ge c_1$.
There are two possibilities, either
\begin{equation}\label{eq2.41}
\frac{\min\{\vert I_{1,1}\vert,\vert I_{1,2}\vert\}}{\vert I_1\vert}<\frac{c_1\eps}{10(A+1)b},
\end{equation}
or
\textcolor{white}{xxxxxxxxxxxxxxxxxxxxxxxxxxxxxx}%%%%%%%%%%
\begin{equation}\label{eq2.42}
\frac{\min\{\vert I_{1,1}\vert,\vert I_{1,2}\vert\}}{\vert I_1\vert}\ge\frac{c_1\eps}{10(A+1)b}.
\end{equation}
If \eqref{eq2.41} holds, then we may assume without loss of generality that
\begin{equation}\label{eq2.43}
\vert I_{1,1}\vert\ge\vert I_{1,2}\vert,
\quad\mbox{so that}\quad
\frac{\vert I_{1,1}\vert}{\vert I_1\vert}\ge\frac{1}{3}
\quad\mbox{and}\quad
\frac{\vert I_{1,2}\vert}{\vert I_1\vert}<\frac{c_1\eps}{10(A+1)b}.
\end{equation}
Combining \eqref{eq2.1}, \eqref{eq2.5} and \eqref{eq2.43}, we obtain
\begin{displaymath}
\vert I_{1,2}\cap\XXX_n\vert
\le2(A+1)n\vert I_{1,2}\vert+2
\le\frac{c_1\eps}{5b}n\vert I_1\vert+2
\le\frac{c_1\eps}{4b}n\vert I_1\vert
\le\frac{c_1\eps}{4}\vert I_1\cap\XXX_n\vert,
\end{displaymath}
provided that
\textcolor{white}{xxxxxxxxxxxxxxxxxxxxxxxxxxxxxx}%%%%%%%%%%
\begin{equation}\label{eq2.44}
2\le\frac{c_1\eps}{20b}n\vert I_1\vert=\frac{c_1\eps C}{20b},
\quad\mbox{or}\quad
C\ge\frac{40b}{c_1\eps}.
\end{equation}
Substituting this into \eqref{eq2.40}, we deduce that
\begin{equation}\label{eq2.45}
\vert I_{1,1}\cap\XXX_n\vert
\ge\vert I_1\cap\XXX_n\vert\left(\frac{\vert I_{1,1}\vert}{\vert I_1\vert}+\frac{c_1\eps}{4}\right).
\end{equation}
On the other hand, if \eqref{eq2.42} holds, then since it follows from \eqref{eq2.40} that
\begin{displaymath}
\vert I_{1,1}\cap\XXX_n\vert+\vert I_{1,2}\cap\XXX_n\vert
\ge\vert I_1\cap\XXX_n\vert\left(\frac{\vert I_{1,1}\vert}{\vert I_1\vert}+\frac{c_1\eps}{4}\right)
+\vert I_1\cap\XXX_n\vert\left(\frac{\vert I_{1,2}\vert}{\vert I_1\vert}+\frac{c_1\eps}{4}\right),
\end{displaymath}
we may assume without loss of generality that \eqref{eq2.45} holds again.
Note now that
\begin{equation}\label{eq2.46}
\vert I_{1,1}\vert\ge\frac{c_1\eps}{10(A+1)b}\vert I_1\vert,
\end{equation}
provided that $\epsilon$ is sufficiently small, and \eqref{eq2.45} leads to the inequality
\begin{equation}\label{eq2.47}
\frac{\vert I_{1,1}\cap\XXX_n\vert}{\vert I_{1,1}\vert}
\ge\frac{\vert I_1\cap\XXX_n\vert}{\vert I_1\vert}\left(1+\frac{c_1\eps}{4}\frac{\vert I_1\vert}{\vert I_{1,1}\vert}\right)
\ge\left(1+\frac{c_1\eps}{4}\right)\frac{\vert I_1\cap\XXX_n\vert}{\vert I_1\vert}.
\end{equation}
Thus the switch from $I_1$ to $I_{1,1}$ leads to density increase by a factor $1+c_1\eps/4$.

The ratio $\vert I_1\vert/\vert I_{1,1}\vert$ is not necessarily an integer.
To overcome this issue, we shall replace $I_{1,1}$ by a suitable subinterval at the expense of part of the density increase.

Let $h_0$ be the unique integer satisfying
\begin{equation}\label{eq2.48}
h_0-1<\frac{48(A+1)b\vert I_1\vert}{c_1\eps\vert I_{1,1}\vert}\le h_0,
\end{equation}
so that
\textcolor{white}{xxxxxxxxxxxxxxxxxxxxxxxxxxxxxx}%%%%%%%%%%
\begin{equation}\label{eq2.49}
\frac{\vert I_1\vert}{h_0}\le\frac{c_1\eps}{48(A+1)b}\vert I_{1,1}\vert\le\vert I_{1,1}\vert,
\end{equation}
provided that $\epsilon$ is sufficiently small.

We now apply Lemma~\ref{lem3}(ii) with $I=I_{1,1}$, $\YYY=I_{1,1}\cap\XXX_n$ and $z=\vert I_1\vert/h_0$.
Note that in view of \eqref{eq2.5}, for every subinterval $I^\dagger\subset I_{1,1}$, we have
\begin{equation}\label{eq2.50}
\vert I^\dagger\cap\YYY\vert
=\vert I^\dagger\cap\XXX_n\vert
\le2(A+1)n\vert I^\dagger\vert+2.
\end{equation}
On the other hand, it follows from \eqref{eq2.1}, \eqref{eq2.47} and $\vert I_1\vert=C/n$ that
\begin{equation}\label{eq2.51}
\frac{\vert I_{1,1}\cap\XXX_n\vert}{\vert I_{1,1}\vert}
\ge\frac{\vert I_1\cap\XXX_n\vert}{\vert I_1\vert}
\ge\frac{n}{b}.
\end{equation}
Combining \eqref{eq2.50} and \eqref{eq2.51}, we have
\begin{displaymath}
\vert I^\dagger\cap\YYY\vert
\le2(A+1)b\vert I_{1,1}\cap\XXX_n\vert\frac{\vert I^\dagger\vert}{\vert I_{1,1}\vert}+2,
\end{displaymath}
so that Lemma~\ref{lem3}(ii) is valid with the constant $B=2(A+1)b$.
It follows that there exists a subinterval $I_1(\star)\subset I_{1,1}$ with $\vert I_1(\star)\vert=z$ such that
\begin{displaymath}
\vert I_1(\star)\cap\XXX_n\vert
\ge(\vert I_{1,1}\cap\XXX_n\vert-2)\frac{z}{\vert I_{1,1}\vert}-2(A+1)b\vert I_{1,1}\cap\XXX_n\vert\left(\frac{z}{\vert I_{1,1}\vert}\right)^2.
\end{displaymath}
Combining this with \eqref{eq2.49}, we have
\begin{align}\label{eq2.52}
\frac{\vert I_1(\star)\cap\XXX_n\vert}{\vert I_1(\star)\vert}
&
\ge\frac{\vert I_{1,1}\cap\XXX_n\vert}{\vert I_{1,1}\vert}\left(1-\frac{2}{\vert I_{1,1}\cap\XXX_n\vert}
-\frac{2(A+1)b}{h_0}\frac{\vert I_1\vert}{\vert I_{1,1}\vert}\right)
\nonumber
\\
&
\ge\frac{\vert I_{1,1}\cap\XXX_n\vert}{\vert I_{1,1}\vert}\left(1-\frac{2}{\vert I_{1,1}\cap\XXX_n\vert}-\frac{c_1\eps}{24}\right).
\end{align}
Next, combining \eqref{eq2.46} and \eqref{eq2.51}, and recalling that $\vert I_1\vert=C/n$, we obtain
\begin{displaymath}
\vert I_{1,1}\cap\XXX_n\vert\ge\frac{c_1C\eps}{10(A+1)b^2}.
\end{displaymath}
We want the bound
\textcolor{white}{xxxxxxxxxxxxxxxxxxxxxxxxxxxxxx}%%%%%%%%%%
\begin{equation}\label{eq2.53}
\frac{2}{\vert I_{1,1}\cap\XXX_n\vert}\le\frac{c_1\eps}{24},
\end{equation}
and this can be guaranteed if we ensure that
\begin{equation}\label{eq2.54}
C\ge\frac{480(A+1)b^2}{(c_1\eps)^2}.
\end{equation}
Combining \eqref{eq2.52} and \eqref{eq2.53}, we now obtain
\begin{displaymath}
\frac{\vert I_1(\star)\cap\XXX_n\vert}{\vert I_1(\star)\vert}
\ge\frac{\vert I_{1,1}\cap\XXX_n\vert}{\vert I_{1,1}\vert}\left(1-\frac{c_1\eps}{12}\right).
\end{displaymath}
Combining this with \eqref{eq2.47}, we deduce that
\begin{equation}\label{eq2.55}
\frac{\vert I_1(\star)\cap\XXX_n\vert}{\vert I_1(\star)\vert}
\ge\left(1+\frac{c_1\eps}{4}\right)\left(1-\frac{c_1\eps}{12}\right)\frac{\vert I_1\cap\XXX_n\vert}{\vert I_1\vert}
\ge\left(1+\frac{c_1\eps}{12}\right)\frac{\vert I_1\cap\XXX_n\vert}{\vert I_1\vert},
\end{equation}
provided that $\epsilon$ is sufficiently small.
Thus the switch from $I_1$ to $I_1(\star)$ leads to density increase by a factor $1+c_1\eps/12$, with the added benefit that the ratio
$\vert I_1\vert/\vert I_1(\star)\vert$ is an integer~$h_0$.

To obtain a subinterval of $I_0$ of the same length as $I_1(\star)$, we next divide $I_0$ into $h_0$ equal parts,
and denote by $I_0(\star)$ one of these subintervals with the minimum intersection with the set~$\XXX_n$.
Then
\begin{equation}\label{eq2.56}
\frac{\vert I_0(\star)\cap\XXX_n\vert}{\vert I_0(\star)\vert}\le\frac{\vert I_0\cap\XXX_n\vert}{\vert I_0\vert}.
\end{equation}
Note that
\textcolor{white}{xxxxxxxxxxxxxxxxxxxxxxxxxxxxxx}%%%%%%%%%%
\begin{equation}\label{eq2.57}
\vert I_0(\star)\vert=\vert I_1(\star)\vert=\frac{C_1}{n},
\quad
C_1=\frac{C}{h_0},
\quad
h_0<c_2,
\end{equation}
where the constant $c_2=c_2(\eps)>0$ is independent of $n$ and~$C$.

%%%%%%%%%%
%
% SECTION 3
%
%%%%%%%%%%

\section{Iteration process: step 1}\label{sec3}

The reader may have observed that we have used the inequality \eqref{eq2.3}, which corresponds to Case~B,
in the argument in Case~1 in the previous section, but not in Case~2.
We now discuss the iterative process that arises from Case~B.

Let $\III_n(\PPP;C_1)$ denote the collection of any subinterval $I$ of any vertical edge of $\PPP$
with length $\vert I\vert=C_1/n$, and let $I_0^{(1)},I_1^{(1)}\in\III_n(\PPP;C_1)$ be subintervals satisfying
\begin{displaymath}
V_n(I_0^{(1)})=\min_{I\in\III_n(\PPP;C_1)}\vert I\cap\XXX_n\vert
\quad\mbox{and}\quad
V_n(I_1^{(1)})=\max_{I\in\III_n(\PPP;C_1)}\vert I\cap\XXX_n\vert,
\end{displaymath}
so that $I_0^{(1)}$ and $I_1^{(1)}$ have respectively the smallest and largest visiting numbers with respect to $\XXX_n$
among all the subintervals $I$ under consideration.
It is clear that
\begin{displaymath}
\vert I_0^{(1)}\cap\XXX_n\vert\le\frac{C_1}{b}\le\vert I_1^{(1)}\cap\XXX_n\vert.
\end{displaymath}
Furthermore, it either, in Case~1, follows from \eqref{eq2.36}--\eqref{eq2.38} that
\begin{equation}\label{eq3.1}
\frac{\vert I_0^{(1)}\cap\XXX_n\vert}{\vert I_0^{(1)}\vert}
\le\frac{\vert I_0(\star)\cap\XXX_n\vert}{\vert I_0(\star)\vert}
\le\left(1-\frac{c_1\eps}{16}\right)\frac{\vert I_0\cap\XXX_n\vert}{\vert I_0\vert}
\end{equation}
and
\textcolor{white}{xxxxxxxxxxxxxxxxxxxxxxxxxxxxxx}%%%%%%%%%%
\begin{equation}\label{eq3.2}
\frac{\vert I_1^{(1)}\cap\XXX_n\vert}{\vert I_1^{(1)}\vert}
\ge\frac{\vert I_1(\star)\cap\XXX_n\vert}{\vert I_1(\star)\vert}
\ge\frac{\vert I_1\cap\XXX_n\vert}{\vert I_1\vert},
\end{equation}
or, in Case~2, follows from \eqref{eq2.55}--\eqref{eq2.57} that
\begin{equation}\label{eq3.3}
\frac{\vert I_0^{(1)}\cap\XXX_n\vert}{\vert I_0^{(1)}\vert}
\le\frac{\vert I_0(\star)\cap\XXX_n\vert}{\vert I_0(\star)\vert}
\le\frac{\vert I_0\cap\XXX_n\vert}{\vert I_0\vert}
\end{equation}
and
\textcolor{white}{xxxxxxxxxxxxxxxxxxxxxxxxxxxxxx}%%%%%%%%%%
\begin{equation}\label{eq3.4}
\frac{\vert I_1^{(1)}\cap\XXX_n\vert}{\vert I_1^{(1)}\vert}
\ge\frac{\vert I_1(\star)\cap\XXX_n\vert}{\vert I_1(\star)\vert}
\ge\left(1+\frac{c_1\eps}{12}\right)\frac{\vert I_1\cap\XXX_n\vert}{\vert I_1\vert}.
\end{equation}

We now concentrate on Case~B, so that the inequality \eqref{eq2.3} holds.
Combining this with \eqref{eq3.1} and \eqref{eq3.2}, or with \eqref{eq3.3} and \eqref{eq3.4}, we obtain
\begin{equation}\label{eq3.5}
\frac{\vert I_0^{(1)}\cap\XXX_n\vert}{\vert I_1^{(1)}\cap\XXX_n\vert}<1-\eps.
\end{equation}

\begin{remark}
Note that \eqref{eq3.5} is the analog of \eqref{eq2.3} and Case~B in Step~0.
It follows that if Case~B in Step~0 holds, then there is no analog of Case~A in Step~1.
\end{remark}

Repeating the argument in Step~0 between \eqref{eq2.10} and \eqref{eq2.15} with $I_0,I_1,C$ replaced by $I_0^{(1)},I_1^{(1)},C_1$ respectively,
we obtain subintervals $I_0^{(1)}(0)\subset I_0^{(1)}$ and $I_1^{(1)}(0)\subset I_1^{(1)}$ such that
\textcolor{white}{xxxxxxxxxxxxxxxxxxxxxxxxxxxxxx}%%%%%%%%%%
\begin{equation}\label{eq3.6}
\frac{C_1}{3n}
=\frac{\vert I_0^{(1)}\vert}{3}
\ge\vert I_0^{(1)}(0)\vert
=\vert I_1^{(1)}(0)\vert
\ge c_1\vert I_0^{(1)}\vert,
\end{equation}
the analog of \eqref{eq2.14}.

We have two cases:

\begin{case1}
We have
\textcolor{white}{xxxxxxxxxxxxxxxxxxxxxxxxxxxxxx}%%%%%%%%%%
\begin{equation}\label{eq3.7}
\frac{\vert I^{(1)}_1(0)\cap\XXX_n\vert}{\vert I^{(1)}_1(0)\vert}
\ge\left(1-\frac{\eps}{2}\right)\frac{\vert I_1^{(1)}\cap\XXX_n\vert}{\vert I_1^{(1)}\vert}.
\end{equation}
\end{case1}

\begin{case2}
We have
\textcolor{white}{xxxxxxxxxxxxxxxxxxxxxxxxxxxxxx}%%%%%%%%%%
\begin{equation}\label{eq3.8}
\frac{\vert I^{(1)}_1(0)\cap\XXX_n\vert}{\vert I^{(1)}_1(0)\vert}
<\left(1-\frac{\eps}{2}\right)\frac{\vert I_1^{(1)}\cap\XXX_n\vert}{\vert I_1^{(1)}\vert}.
\end{equation}
\end{case2}

%%%%%%%%%%
%
% SECTION 3.1
%
%%%%%%%%%%

\subsection{Case 1: density decrease}\label{sec31}

Suppose that the inequality \eqref{eq3.7} holds.
Then an argument analogous to that in Step~0 between \eqref{eq2.21} and \eqref{eq2.23} now leads to the inequality
\textcolor{white}{xxxxxxxxxxxxxxxxxxxxxxxxxxxxxx}%%%%%%%%%%
\begin{equation}\label{eq3.9}
\vert I^{(1)}_0(0)\cap\XXX_n\vert
\ge\left(1-\frac{3\eps}{4}\right)\vert I^{(1)}_0(0)\vert\frac{\vert I_1^{(1)}\cap\XXX_n\vert}{\vert I_1^{(1)}\vert},
\end{equation}
provided that
\textcolor{white}{xxxxxxxxxxxxxxxxxxxxxxxxxxxxxx}%%%%%%%%%%
\begin{equation}\label{eq3.10}
C_1\ge\frac{96(A+1)^2c_0b}{\eps(1+\alpha^2)^{1/2}}
\quad\mbox{and}\quad
C_1\ge\frac{16b}{c_1\eps}.
\end{equation}

Corresponding to \eqref{eq2.26}, there are at most two subintervals $I_{0,1}^{(1)},I_{0,2}^{(1)}\subset I_0^{(1)}$ such that
\textcolor{white}{xxxxxxxxxxxxxxxxxxxxxxxxxxxxxx}%%%%%%%%%%
\begin{displaymath}
I_0^{(1)}=I^{(1)}_0(0)\cup I_{0,1}^{(1)}\cup I_{0,2}^{(1)}.
\end{displaymath}
An argument analogous to that in Step~0 between \eqref{eq2.26} and \eqref{eq2.31} then shows that, without loss of generality,
\textcolor{white}{xxxxxxxxxxxxxxxxxxxxxxxxxxxxxx}%%%%%%%%%%
\begin{displaymath}
\vert I_{0,1}^{(1)}\vert\ge\frac{c_1\eps}{8}\vert I_0^{(1)}\vert,
\end{displaymath}
as well as
\textcolor{white}{xxxxxxxxxxxxxxxxxxxxxxxxxxxxxx}%%%%%%%%%%
\begin{displaymath}
\frac{\vert I_{0,1}^{(1)}\cap\XXX_n\vert}{\vert I_{0,1}^{(1)}\vert}
\le\left(1-\frac{c_1\eps}{8}\right)\frac{\vert I_0^{(1)}\cap\XXX_n\vert}{\vert I_0^{(1)}\vert}.
\end{displaymath}
An argument analogous to that in Step~0 between \eqref{eq2.33} and \eqref{eq2.36} then leads to the existence of a subinterval
$I_0^{(1)}(\star)\subset I_0^{(1)}$ satisfying $\vert I_0^{(1)}\vert/\vert I_0^{(1)}(\star)\vert=h_1$, where $h_1$ is the unique integer satisfying
\begin{displaymath}
h_1-1<\frac{16\vert I_0^{(1)}\vert}{c_1\eps\vert I_{0,1}^{(1)}\vert}\le h_1,
\end{displaymath}
such that
\textcolor{white}{xxxxxxxxxxxxxxxxxxxxxxxxxxxxxx}%%%%%%%%%%
\begin{equation}\label{eq3.11}
\frac{\vert I_0^{(1)}(\star)\cap\XXX_n\vert}{\vert I_0^{(1)}(\star)\vert}
\le\left(1-\frac{c_1\eps}{16}\right)\frac{\vert I_0^{(1)}\cap\XXX_n\vert}{\vert I_0^{(1)}\vert}.
\end{equation}

To obtain a subinterval of $I_1^{(1)}$ of the same length as $I_0^{(1)}(\star)$, we next divide $I_1^{(1)}$ into $h_1$ equal parts, and denote by
$I_1^{(1)}(\star)$ one of these subintervals with the maximum intersection with the set~$\XXX_n$.
Then
\begin{displaymath}
\frac{\vert I_1^{(1)}(\star)\cap\XXX_n\vert}{\vert I_1^{(1)}(\star)\vert}\ge\frac{\vert I_1^{(1)}\cap\XXX_n\vert}{\vert I_1^{(1)}\vert}.
\end{displaymath}
Note that
\textcolor{white}{xxxxxxxxxxxxxxxxxxxxxxxxxxxxxx}%%%%%%%%%%
\begin{equation}\label{eq3.12}
\vert I_1^{(1)}(\star)\vert=\vert I_0^{(1)}(\star)\vert=\frac{C_2}{n},
\quad
C_2=\frac{C_1}{h_1},
\quad
h_1<c_2,
\end{equation}
where the constant $c_2=c_2(\eps)>0$ is as in Step~0.

%%%%%%%%%%
%
% SECTION 3.2
%
%%%%%%%%%%

\subsection{Case 2: density increase}\label{sec32}

Suppose that the inequality \eqref{eq3.8} holds.
Then corresponding to \eqref{eq2.39}, there are at most two subintervals $I_{1,1}^{(1)},I_{1,2}^{(1)}\subset I_1^{(1)}$ such that
\begin{equation}\label{eq3.13}
I_1^{(1)}=I^{(1)}_1(0)\cup I_{1,1}^{(1)}\cup I_{1,2}^{(1)}.
\end{equation}
An argument analogous to that in Step~0 between \eqref{eq2.39} and \eqref{eq2.47} then shows that, without loss of generality,
\textcolor{white}{xxxxxxxxxxxxxxxxxxxxxxxxxxxxxx}%%%%%%%%%%
\begin{displaymath}
\vert I_{1,1}^{(1)}\vert\ge\frac{c_1\eps}{10(A+1)b}\vert I_1^{(1)}\vert,
\end{displaymath}
as well as
\textcolor{white}{xxxxxxxxxxxxxxxxxxxxxxxxxxxxxx}%%%%%%%%%%
\begin{displaymath}
\frac{\vert I_{1,1}^{(1)}\cap\XXX_n\vert}{\vert I_{1,1}^{(1)}\vert}
\ge\left(1+\frac{c_1\eps}{4}\right)\frac{\vert I_1^{(1)}\cap\XXX_n\vert}{\vert I_1^{(1)}\vert}.
\end{displaymath}
An argument analogous to that in Step~0 between \eqref{eq2.48} and \eqref{eq2.55} then leads to the existence of a subinterval
$I_1^{(1)}(\star)\subset I_1^{(1)}$ satisfying $\vert I_1^{(1)}\vert/\vert I_1^{(1)}(\star)\vert=h_1$, where $h_1$ is the unique integer satisfying
\begin{displaymath}
h_1-1<\frac{48(A+1)b\vert I_1^{(1)}\vert}{c_1\eps\vert I_{1,1}^{(1)}\vert}\le h_1,
\end{displaymath}
such that
\textcolor{white}{xxxxxxxxxxxxxxxxxxxxxxxxxxxxxx}%%%%%%%%%%
\begin{displaymath}
\frac{\vert I_1^{(1)}(\star)\cap\XXX_n\vert}{\vert I_1^{(1)}(\star)\vert}
\ge\left(1+\frac{c_1\eps}{12}\right)\frac{\vert I_1^{(1)}\cap\XXX_n\vert}{\vert I_1^{(1)}\vert},
\end{displaymath}
provided that
\textcolor{white}{xxxxxxxxxxxxxxxxxxxxxxxxxxxxxx}%%%%%%%%%%
\begin{equation}\label{eq3.14}
C_1\ge\frac{40b}{c_1\eps}
\quad\mbox{and}\quad
C_1\ge\frac{480(A+1)b^2}{(c_1\eps)^2}.
\end{equation}

To obtain a subinterval of $I_0^{(1)}$ of the same length as $I_1^{(1)}(\star)$, we next divide $I_0^{(1)}$ into $h_1$ equal parts,
and denote by $I_0^{(1)}(\star)$ one of these subintervals with the minimum intersection with the set~$\XXX_n$.
Then
\begin{displaymath}
\frac{\vert I_0^{(1)}(\star)\cap\XXX_n\vert}{\vert I_0^{(1)}(\star)\vert}\le\frac{\vert I_0^{(1)}\cap\XXX_n\vert}{\vert I_0^{(1)}\vert}.
\end{displaymath}
Note that
\textcolor{white}{xxxxxxxxxxxxxxxxxxxxxxxxxxxxxx}%%%%%%%%%%
\begin{equation}\label{eq3.15}
\vert I_0^{(1)}(\star)\vert=\vert I_1^{(1)}(\star)\vert=\frac{C_2}{n},
\quad
C_2=\frac{C_1}{h_1},
\quad
h_1<c_2,
\end{equation}
where the constant $c_2=c_2(\eps)>0$ is as in Step~0.

%%%%%%%%%%
%
% SECTION 4
%
%%%%%%%%%%

\section{Iteration process: general step}\label{sec4}

Let $\III_n(\PPP;C_i)$ denote the collection of any subinterval $I$ of any vertical edge of $\PPP$ with length $\vert I\vert=C_i/n$, and let
$I_0^{(i)},I_1^{(i)}\in\III_n(\PPP;C_i)$ be subintervals satisfying
\begin{displaymath}
V_n(I_0^{(i)})=\min_{I\in\III_n(\PPP;C_i)}\vert I\cap\XXX_n\vert
\quad\mbox{and}\quad
V_n(I_1^{(i)})=\max_{I\in\III_n(\PPP;C_i)}\vert I\cap\XXX_n\vert,
\end{displaymath}
so that $I_0^{(i)}$ and $I_1^{(i)}$ have respectively the smallest and largest visiting numbers with respect to $\XXX_n$
among all the subintervals $I$ under consideration.
It is clear that
\begin{displaymath}
\vert I_0^{(i)}\cap\XXX_n\vert\le\frac{C_i}{b}\le\vert I_1^{(i)}\cap\XXX_n\vert.
\end{displaymath}
Furthermore, we either, in Case~1 in the previous step and analogous to \eqref{eq3.1} and \eqref{eq3.2}, have
\textcolor{white}{xxxxxxxxxxxxxxxxxxxxxxxxxxxxxx}%%%%%%%%%%
\begin{equation}\label{eq4.1}
\frac{\vert I_0^{(i)}\cap\XXX_n\vert}{\vert I_0^{(i)}\vert}
\le\left(1-\frac{c_1\eps}{16}\right)\frac{\vert I_0^{(i-1)}\cap\XXX_n\vert}{\vert I_0^{(i-1)}\vert}
\end{equation}
and
\textcolor{white}{xxxxxxxxxxxxxxxxxxxxxxxxxxxxxx}%%%%%%%%%%
\begin{equation}\label{eq4.2}
\frac{\vert I_1^{(i)}\cap\XXX_n\vert}{\vert I_1^{(i)}\vert}
\ge\frac{\vert I_1^{(i-1)}\cap\XXX_n\vert}{\vert I_1^{(i-1)}\vert},
\end{equation}
or, in Case~2 in the previous step and analogous to \eqref{eq3.3} and \eqref{eq3.4}, have
\begin{equation}\label{eq4.3}
\frac{\vert I_0^{(i)}\cap\XXX_n\vert}{\vert I_0^{(i)}\vert}
\le\frac{\vert I_0^{(i-1)}\cap\XXX_n\vert}{\vert I_0^{(i-1)}\vert}
\end{equation}
and
\textcolor{white}{xxxxxxxxxxxxxxxxxxxxxxxxxxxxxx}%%%%%%%%%%
\begin{equation}\label{eq4.4}
\frac{\vert I_1^{(i)}\cap\XXX_n\vert}{\vert I_1^{(i)}\vert}
\ge\left(1+\frac{c_1\eps}{12}\right)\frac{\vert I_1^{(i-1)}\cap\XXX_n\vert}{\vert I_1^{(i-1)}\vert}.
\end{equation}
Combining the estimate
\textcolor{white}{xxxxxxxxxxxxxxxxxxxxxxxxxxxxxx}%%%%%%%%%%
\begin{displaymath}
\frac{\vert I_0^{(i-1)}\cap\XXX_n\vert}{\vert I_1^{(i-1)}\cap\XXX_n\vert}<1-\eps
\end{displaymath}
from the previous step with \eqref{eq4.1} and \eqref{eq4.2}, or with \eqref{eq4.3} and \eqref{eq4.4}, we obtain
\begin{displaymath}
\frac{\vert I_0^{(i)}\cap\XXX_n\vert}{\vert I_1^{(i)}\cap\XXX_n\vert}<1-\eps,
\end{displaymath}
the analog of \eqref{eq2.3} and \eqref{eq3.5}.

On the other hand, iterating \eqref{eq4.1}--\eqref{eq4.4} carefully, we obtain
\begin{equation}\label{eq4.5}
\frac{\vert I_0^{(i)}\cap\XXX_n\vert}{\vert I_0^{(i)}\vert}
\le\left(1-\frac{c_1\eps}{16}\right)^{i_1}\frac{\vert I_0\cap\XXX_n\vert}{\vert I_0\vert}
\end{equation}
and
\textcolor{white}{xxxxxxxxxxxxxxxxxxxxxxxxxxxxxx}%%%%%%%%%%
\begin{equation}\label{eq4.6}
\frac{\vert I_1^{(i)}\cap\XXX_n\vert}{\vert I_1^{(i)}\vert}
\ge\left(1+\frac{c_1\eps}{12}\right)^{i_2}\frac{\vert I_1\cap\XXX_n\vert}{\vert I_1\vert},
\end{equation}
where $i_1$ and $i_2$ denote respectively the number of times Case~1 and Case~2 are valid in the previous $i=i_1+i_2$ steps.
Combining \eqref{eq2.1}, \eqref{eq4.5} and \eqref{eq4.6}, and recalling that $\vert I_0\vert=\vert I_1\vert$
and $\vert I_0^{(i)}\vert=\vert I_1^{(i)}\vert$, we obtain the inequality
\begin{equation}\label{eq4.7}
\vert I_0^{(i)}\cap\XXX_n\vert\le\left(1-\frac{c_1\eps}{16}\right)^{i_1}\vert I_1^{(i)}\cap\XXX_n\vert.
\end{equation}

Repeating the argument in Step~0 between \eqref{eq2.10} and \eqref{eq2.15} with $I_0,I_1,C$ replaced by $I_0^{(i)},I_1^{(i)},C_i$ respectively,
we obtain subintervals $I_0^{(1)}(0)\subset I_0^{(1)}$ and $I_1^{(1)}(0)\subset I_1^{(1)}$ such that
\textcolor{white}{xxxxxxxxxxxxxxxxxxxxxxxxxxxxxx}%%%%%%%%%%
\begin{equation}\label{eq4.8}
\frac{C_i}{3n}
=\frac{\vert I_0^{(i)}\vert}{3}
\ge\vert I_0^{(i)}(0)\vert
=\vert I_1^{(i)}(0)\vert
\ge c_1\vert I_0^{(i)}\vert,
\end{equation}
the analog of \eqref{eq2.14} and \eqref{eq3.6},
where the constant $c_1=c_1(\PPP;\alpha)$ in \eqref{eq4.8} is exactly the same as before.

We have two cases:

\begin{case1}
We have
\textcolor{white}{xxxxxxxxxxxxxxxxxxxxxxxxxxxxxx}%%%%%%%%%%
\begin{equation}\label{eq4.9}
\frac{\vert I^{(i)}_1(0)\cap\XXX_n\vert}{\vert I^{(i)}_1(0)\vert}
\ge\left(1-\frac{\eps}{2}\right)\frac{\vert I_1^{(i)}\cap\XXX_n\vert}{\vert I_1^{(i)}\vert}.
\end{equation}
\end{case1}

\begin{case2}
We have
\textcolor{white}{xxxxxxxxxxxxxxxxxxxxxxxxxxxxxx}%%%%%%%%%%
\begin{equation}\label{eq4.10}
\frac{\vert I^{(i)}_1(0)\cap\XXX_n\vert}{\vert I^{(i)}_1(0)\vert}
<\left(1-\frac{\eps}{2}\right)\frac{\vert I_1^{(i)}\cap\XXX_n\vert}{\vert I_1^{(i)}\vert}.
\end{equation}
\end{case2}

%%%%%%%%%%
%
% SECTION 4.1
%
%%%%%%%%%%

\subsection{Case 1: density decrease}\label{sec41}

Suppose that the inequality \eqref{eq4.9} holds.
Then an argument analogous to that in Step~0 between \eqref{eq2.21} and \eqref{eq2.36} and that in Step~1
between \eqref{eq3.9} and \eqref{eq3.11} leads to the existence of a subinterval $I_0^{(i)}(\star)\subset I_0^{(i)}$
satisfying $\vert I_0^{(i)}\vert/\vert I_0^{(i)}(\star)\vert=h_i$, where $h_i$ is an integer and
\begin{displaymath}
\frac{\vert I_0^{(i)}(\star)\cap\XXX_n\vert}{\vert I_0^{(i)}(\star)\vert}
\le\left(1-\frac{c_1\eps}{16}\right)\frac{\vert I_0^{(i)}\cap\XXX_n\vert}{\vert I_0^{(i)}\vert}.
\end{displaymath}
provided that
\textcolor{white}{xxxxxxxxxxxxxxxxxxxxxxxxxxxxxx}%%%%%%%%%%
\begin{equation}\label{eq4.11}
C_i\ge\frac{96(A+1)^2c_0b}{\eps(1+\alpha^2)^{1/2}}
\quad\mbox{and}\quad
C_i\ge\frac{16b}{c_1\eps},
\end{equation}

To obtain a subinterval of $I_1^{(i)}$ of the same length as $I_0^{(i)}(\star)$, we next divide $I_1^{(i)}$ into $h_i$ equal parts,
and denote by $I_1^{(i)}(\star)$ one of these subintervals with the maximum intersection with the set~$\XXX_n$.
Then
\begin{displaymath}
\frac{\vert I_1^{(i)}(\star)\cap\XXX_n\vert}{\vert I_1^{(i)}(\star)\vert}\ge\frac{\vert I_1^{(i)}\cap\XXX_n\vert}{\vert I_1^{(i)}\vert}.
\end{displaymath}
We have
\textcolor{white}{xxxxxxxxxxxxxxxxxxxxxxxxxxxxxx}%%%%%%%%%%
\begin{equation}\label{eq4.12}
\vert I_1^{(i)}(\star)\vert=\vert I_0^{(i)}(\star)\vert=\frac{C_{i+1}}{n},
\quad
C_{i+1}=\frac{C_i}{h_i},
\quad
h_i<c_2,
\end{equation}
where the constant $c_2=c_2(\eps)>0$ is as in Step~0.

%%%%%%%%%%
%
% SECTION 4.2
%
%%%%%%%%%%

\subsection{Case 2: density increase}\label{sec42}

Suppose that the inequality \eqref{eq4.10} holds.
Then an argument analogous to that in Step~0 between \eqref{eq2.39} and \eqref{eq2.55} and that in Step~1
between \eqref{eq3.13} and \eqref{eq3.14} leads to the existence of a subinterval $I_1^{(i)}(\star)\subset I_1^{(i)}$
satisfying $\vert I_1^{(i)}\vert/\vert I_1^{(i)}(\star)\vert=h_i$, where $h_i$ is an integer and
\begin{displaymath}
\frac{\vert I_1^{(i)}(\star)\cap\XXX_n\vert}{\vert I_1^{(i)}(\star)\vert}
\ge\left(1+\frac{c_1\eps}{12}\right)\frac{\vert I_1^{(i)}\cap\XXX_n\vert}{\vert I_1^{(i)}\vert},
\end{displaymath}
provided that
\textcolor{white}{xxxxxxxxxxxxxxxxxxxxxxxxxxxxxx}%%%%%%%%%%
\begin{equation}\label{eq4.13}
C_i\ge\frac{40b}{c_1\eps}
\quad\mbox{and}\quad
C_i\ge\frac{480(A+1)b^2}{(c_1\eps)^2}.
\end{equation}

To obtain a subinterval of $I_0^{(i)}$ of the same length as $I_1^{(i)}(\star)$, we next divide $I_0^{(i)}$ into $h_i$ equal parts,
and denote by $I_0^{(i)}(\star)$ one of these subintervals with the minimum intersection with the set~$\XXX_n$.
Then
\begin{displaymath}
\frac{\vert I_0^{(i)}(\star)\cap\XXX_n\vert}{\vert I_0^{(i)}(\star)\vert}\le\frac{\vert I_0^{(i)}\cap\XXX_n\vert}{\vert I_0^{(i)}\vert}.
\end{displaymath}
We have
\textcolor{white}{xxxxxxxxxxxxxxxxxxxxxxxxxxxxxx}%%%%%%%%%%
\begin{equation}\label{eq4.14}
\vert I_0^{(i)}(\star)\vert=\vert I_1^{(i)}(\star)\vert=\frac{C_{i+1}}{n},
\quad
C_{i+1}=\frac{C_i}{h_i},
\quad
h_i<c_2,
\end{equation}
where the constant $c_2=c_2(\eps)>0$ is as in Step~0.

%%%%%%%%%%
%
% SECTION 5
%
%%%%%%%%%%

\section{Iteration process: deriving a contradiction}\label{sec5}

We now attempt to derive the necessary contradiction.

Suppose first that Case~1 holds in Step~$i$.

Corresponding to \eqref{eq2.23} and \eqref{eq3.9}, we have the inequality
\begin{equation}\label{eq5.1}
\vert I^{(i)}_0(0)\cap\XXX_n\vert
\ge\left(1-\frac{3\eps}{4}\right)\vert I^{(i)}_0(0)\vert\frac{\vert I_1^{(i)}\cap\XXX_n\vert}{\vert I_1^{(i)}\vert},
\end{equation}
provided that \eqref{eq4.11} holds.

Clearly $I_0^{(i)}(0)\subset I_0^{(i)}$, so it follows from \eqref{eq4.7} that
\begin{equation}\label{eq5.2}
\vert I_0^{(i)}(0)\cap\XXX_n\vert\le\left(1-\frac{c_1\eps}{16}\right)^{i_1}\vert I_1^{(i)}\cap\XXX_n\vert.
\end{equation}
On the other hand, combining \eqref{eq4.8} with \eqref{eq5.1} leads to the inequality
\begin{equation}\label{eq5.3}
\vert I^{(i)}_0(0)\cap\XXX_n\vert
\ge c_1\left(1-\frac{3\eps}{4}\right)\vert I_1^{(i)}\cap\XXX_n\vert.
\end{equation}
Clearly \eqref{eq5.2} and \eqref{eq5.3} contradict each other if
\begin{displaymath}
\left(1-\frac{c_1\eps}{16}\right)^{i_1}<\frac{c_1}{2},
\end{displaymath}
noting that $0<\eps<1/2$.
This gives an upper bound $c_3=c_3(\eps)$ to~$i_1$, the number of times that Case~1 holds among the first $i$ steps.

Suppose next that Case~2 holds in Step~$i$.

Combining \eqref{eq2.1} and \eqref{eq4.6}, and noting that $\vert I_1\vert=C/n$, we have
\begin{equation}\label{eq5.4}
\vert I_1^{(i)}\cap\XXX_n\vert\ge\left(1+\frac{c_1\eps}{12}\right)^{i_2}\frac{n}{b}\vert I_1^{(i)}\vert.
\end{equation}
On the other hand, it follows from \eqref{eq2.19} that
\begin{equation}\label{eq5.5}
\vert I_1^{(i)}\cap\XXX_n\vert\le2(A+1)n\vert I_1^{(i)}\vert+2.
\end{equation}
Clearly \eqref{eq5.4} and \eqref{eq5.5} contradict each other if 
\begin{displaymath}
\left(1+\frac{c_1\eps}{12}\right)^{i_2}>4(A+1)b,
\end{displaymath}
provided that $C_i\ge1$.
This gives an upper bound $c_4=c_4(\eps)$ to~$i_2$, the number of times that Case~2 holds among the first $i$ steps.

If $i>c_3+c_4$, then neither Case~1 nor Case~2 in Step~$i$ can be valid.
To show that Case~B is impossible, it remains to analyze the various constants in our argument.

Recall that the constants $c_0=c_0(\PPP;\alpha)$ and $c_1=c_1(\PPP;\alpha)$ depend only on $\PPP$ and~$\alpha$,
and are independent of $n$, $C$ and~$\eps$, while the constant $A$ depends only on~$\alpha$,
and the constant $b$ depends only on~$\PPP$.
It remains to study the various constants $C$ and~$C_i$.
They are governed by the inequalities \eqref{eq2.24}, \eqref{eq2.38}, \eqref{eq2.44}, \eqref{eq2.54}, \eqref{eq2.57},
\eqref{eq3.10}, \eqref{eq3.12}, \eqref{eq3.14}, \eqref{eq3.15} and \eqref{eq4.11}--\eqref{eq4.14}.
Thus we need
\begin{equation}\label{eq5.6}
C_i=\frac{C}{h_0h_1,\ldots,h_{i-1}}
\ge\max\left\{\frac{96(A+1)^2c_0b}{\eps(1+\alpha^2)^{1/2}},\frac{40b}{c_1\eps},\frac{480(A+1)b^2}{(c_1\eps)^2},1\right\},
\end{equation}
and $h_i<c_2=c_2(\eps)$.

The iteration process must stop after at most $c_5=c_5(\eps)=c_3(\eps)+c_4(\eps)$ steps.
It follows that \eqref{eq5.6} is satisfied provided that $C$ is chosen sufficiently large in terms of $\PPP,\alpha$ and~$\eps$.

For convenience, let $C^*$ be sufficiently large so that \eqref{eq5.6} is satisfied for every real number $C$ satisfying
\textcolor{white}{xxxxxxxxxxxxxxxxxxxxxxxxxxxxxx}%%%%%%%%%%
\begin{displaymath}
C\ge C^*.
\end{displaymath}
%
%

%%%%%%%%%%
%
% SECTION 6
%
%%%%%%%%%%

\section{Proof of Theorem~\ref{thm1}}\label{sec6}

We have already shown that Case~B leads to a contradiction.
To complete the proof of Theorem~\ref{thm1}, it remains to investigate Case~A, when the inequality \eqref{eq2.2} holds.
We have the following almost trivial observation.

\begin{lem}\label{lem4}
Suppose that $J$ is a subinterval of any vertical edge of $\PPP$ with length $\vert J\vert\ge3C/n$,
where $C$ is an integer satisfying $C^*\le C<n$.
If \eqref{eq2.2} holds, then
\begin{equation}\label{eq6.1}
(1-\eps)\left(\frac{\vert J\vert}{\vert I_1\vert}-3\right)\vert I_1\cap\XXX_n\vert
\le\vert J\cap\XXX_n\vert
\le\left(\frac{\vert J\vert}{\vert I_1\vert}+3\right)\vert I_1\cap\XXX_n\vert.
\end{equation}
\end{lem}

\begin{proof}
Let $k=[n/C]$ denote the integer part of~$n/C$.
Then we can split any vertical edge of $\PPP$ into a union of $k$ special subintervals of length $C/n$
and an extra short interval with length $w$ satisfying $0\le w<C/n$ at the top end of the vertical edge.

Consider the unique integer $\ell_0$ that satisfies the inequalities
\begin{equation}\label{eq6.2}
\ell_0\le\frac{\vert J\vert}{C/n}=\frac{\vert J\vert}{\vert I_1\vert}<\ell_0+1.
\end{equation}
Then $J$ contains at least $\ell_0-2$ of these special subintervals of length~$C/n$.
Combining this observation with \eqref{eq2.2} and the second inequality in \eqref{eq6.2} leads to the lower bound
\begin{displaymath}
\vert J\cap\XXX_n\vert
\ge(\ell_0-2)\vert I_0\cap\XXX_n\vert
>\left(\frac{\vert J\vert}{\vert I_1\vert}-3\right)(1-\eps)\vert I_1\cap\XXX_n\vert.
\end{displaymath}
On the other hand, $J$ is covered by $\ell_0+2$ special subintervals of length $C/n$ and the extra short subinterval of length $w$
which is contained in a subinterval of the vertical edge of length $C/n$.
Combining this observation with the first inequality in \eqref{eq6.2} leads to the upper bound
\begin{displaymath}
\vert J\cap\XXX_n\vert
\le(\ell_0+3)\vert I_1\cap\XXX_n\vert
\le\left(\frac{\vert J\vert}{\vert I_1\vert}+3\right)\vert I_1\cap\XXX_n\vert.
\end{displaymath}
This completes the proof.
\end{proof}

Let $C_\eps=3C^*/\eps$.
Let $J$ be an interval on a vertical edge of $\PPP$ satisfying
\begin{displaymath}
\vert J\vert\ge\frac{C_\eps}{n}=\frac{3C^*}{\eps n}.
\end{displaymath}
Then $\vert J\vert=\LLLL/\eps n$ for some positive real number $\LLLL\in\Rr$.
Clearly there exists an integer $C\ge C^*$ such that $3(C-1)<\LLLL\le3C$, so that
\begin{equation}\label{eq6.3}
\vert J\vert=\frac{3C}{\eps^*n}=\frac{3\vert I_1\vert}{\eps^*}
\end{equation}
for some $\eps^*$ satisfying
\textcolor{white}{xxxxxxxxxxxxxxxxxxxxxxxxxxxxxx}%%%%%%%%%%
\begin{equation}\label{eq6.4}
\eps\le\eps^*\le2\eps.
\end{equation}
Making use of \eqref{eq6.3}, we see that
\begin{align}\label{eq6.5}
&
\left\vert V_n(J)-\frac{n\vert J\vert}{b}\right\vert
=\left\vert\vert J\cap\XXX_n\vert-\frac{n}{b}\frac{3\vert I_1\vert}{\eps^*}\right\vert
\nonumber
\\
&\quad
\le\left\vert\vert J\cap\XXX_n\vert-\frac{3}{\eps^*}\vert I_1\cap\XXX_n\vert\right\vert
+\frac{3\vert I_1\vert}{\eps^*}\left(\frac{\vert I_1\cap\XXX_n\vert}{\vert I_1\vert}-\frac{n}{b}\right).
\end{align}
With $\vert J\vert=3\vert I_1\vert/\eps^*$ in \eqref{eq6.1}, we have
\begin{displaymath}
\frac{3(1-\eps^*)^2}{\eps^*}\vert I_1\cap\XXX_n\vert
\le\vert J\cap\XXX_n\vert
\le\frac{3(1+\eps^*)}{\eps^*}\vert I_1\cap\XXX_n\vert,
\end{displaymath}
and this implies
\textcolor{white}{xxxxxxxxxxxxxxxxxxxxxxxxxxxxxx}%%%%%%%%%%
\begin{equation}\label{eq6.6}
\left\vert\vert J\cap\XXX_n\vert-\frac{3}{\eps^*}\vert I_1\cap\XXX_n\vert\right\vert\le6\vert I_1\cap\XXX_n\vert.
\end{equation}
On the other hand, it is clear from \eqref{eq2.1} and \eqref{eq2.2} that
\begin{equation}\label{eq6.7}
\frac{\vert I_1\cap\XXX_n\vert}{\vert I_1\vert}
\ge\frac{n}{b}
\ge\frac{\vert I_0\cap\XXX_n\vert}{\vert I_0\vert}
\ge(1-\eps^*)\frac{\vert I_1\cap\XXX_n\vert}{\vert I_1\vert}.
\end{equation}
It then follows from \eqref{eq6.7} that
\begin{displaymath}
\frac{\vert I_1\cap\XXX_n\vert}{\vert I_1\vert}-\frac{n}{b}
\le\frac{\vert I_1\cap\XXX_n\vert}{\vert I_1\vert}-\frac{\vert I_0\cap\XXX_n\vert}{\vert I_0\vert}
\le\frac{\vert I_1\cap\XXX_n\vert}{\vert I_1\vert}-(1-\eps^*)\frac{\vert I_1\cap\XXX_n\vert}{\vert I_1\vert},
\end{displaymath}
so that
\textcolor{white}{xxxxxxxxxxxxxxxxxxxxxxxxxxxxxx}%%%%%%%%%%
\begin{equation}\label{eq6.8}
\frac{3\vert I_1\vert}{\eps^*}\left(\frac{\vert I_1\cap\XXX_n\vert}{\vert I_1\vert}-\frac{n}{b}\right)\le3\vert I_1\cap\XXX_n\vert.
\end{equation}
It also follows from \eqref{eq6.7} that
\textcolor{white}{xxxxxxxxxxxxxxxxxxxxxxxxxxxxxx}%%%%%%%%%%
\begin{equation}\label{eq6.9}
\vert I_1\cap\XXX_n\vert\le\frac{\vert I_1\vert}{1-\eps^*}\frac{n}{b}.
\end{equation}
Substituting \eqref{eq6.6}, \eqref{eq6.8} and \eqref{eq6.9} into \eqref{eq6.5}, we conclude that
\begin{equation}\label{eq6.10}
\left\vert V_n(J)-\frac{n\vert J\vert}{b}\right\vert
\le\frac{3\eps^*}{1-\eps^*}\frac{n\vert J\vert}{b}
\le\frac{6\eps}{1-2\eps}\frac{n\vert J\vert}{b},
\end{equation}
in view of \eqref{eq6.4}.
Naturally, we may assume that $\eps<1/2$.
Since $n$ and $J$ are arbitrary, the inequality \eqref{eq6.10} proves super-micro-uniformity with $6\eps(1-2\eps)^{-1}$ instead of~$\eps$.

%%%%%%%%%%
%
% REFERENCES
%
%%%%%%%%%%

\end{document}